\documentclass[11pt]{amsart}
\usepackage{amsbsy,amssymb,amscd,amsfonts,latexsym,amstext,delarray, amsmath,color,caption}
\usepackage{tikz,lscape}
\usetikzlibrary{matrix,arrows}
\usepackage{graphicx}
\input xy

\xyoption{all}
\usepackage{tikz-cd}
    \advance \topmargin by -\headheight
    \advance \topmargin by -\headsep

    \evensidemargin \oddsidemargin
\newtheorem {theorem}    {Theorem}[section]

\newtheorem {lemma}      [theorem]    {Lemma}

\theoremstyle{definition}

\theoremstyle{remark}

\def\Bhat{{\widehat{B}}}
\def\Ghat{{\widehat{G}}}

\def\a{\alpha}

\def\eps{\varepsilon}

\def\lam{\lambda}            
                
\def\phi{\varphi}

\def\a{\alpha}

\def\Z{{\mathbb Z}}
\def\C{{\mathbb C}}

\def\Q{\mathbb{Q}}     
\def\Z{\mathbb{Z}}     
\def\C{\mathbb{C}}     

\def\myscale{.5}
\def\mywidth{.8pt}
\def\vertexradius{.1}
\def\vertex(#1){\fill (#1) circle (\vertexradius)}

\begin{document}


\title{\bf {Uniqueness of representation--theoretic \\hyperbolic Kac--Moody groups over $\Z$}}

\date{\today}

\author{Lisa Carbone and Frank Wagner}

\begin{abstract} For a simply laced and hyperbolic Kac--Moody group $G=G(R)$ over a commutative ring  $R$ with 1, we consider a map from a finite presentation of $G(R)$ obtained by  Allcock and Carbone  to a representation--theoretic construction $G^{\lambda}(R)$ corresponding to an integrable representation $V^{\lambda}$ with dominant integral weight $\lambda$. When $R=\Z$, we prove that this map extends to a group homomorphism
$\rho_{\lambda,\Z}: G(\Z)  \to   G^{\lambda}(\Z).$ We prove that  the kernel $K^{\lambda}$ of  $\rho_{\lam,\Z}$ lies in $H(\C)$ and if  the  natural group homomorphism $\varphi:G(\Z)\to G(\C)$ is injective, then $K^{\lambda}\leq H(\Z)\cong(\Z/2\Z)^{rank(G)}$.
\end{abstract}

\maketitle

\section{Introduction}

Here we consider and compare two distinct constructions of a simply laced and hyperbolic Kac--Moody group $G=G(R)$ over a commutative ring  $R$ with 1. One of these groups is   Tits's construction of $G(R)$, though we work only with a finite presentation of $G(R)$ obtained by  Allcock and Carbone ([AC]). The other is a Kac--Moody group $G^{\lambda}(R)$ constructed in [Ca] (following [CG]) using an integrable representation $V^{\lambda}$ of the underlying Kac--Moody algebra with dominant integral weight $\lambda$.

We consider a natural map between generators of $G(R)$ and  $G^{\lambda}(R)$.  When $R=\Z$, we prove that this map extends to a group homomorphism
$\rho_{\lambda,\Z}:G(\Z)  \to   G^{\lambda}(\Z).$ We prove that  the kernel $K^{\lambda}$ of the map $\rho_{\lam,\Z}: G(\Z)\to G^{\lam}(\Z)$ lies in $H(\C)$ and if  the  natural group homomorphism $\varphi:G(\Z)\to G(\C)$ is injective, then $K^{\lambda}\leq H(\Z)\cong(\Z/2\Z)^{rank(G)}$. Injectivity of the natural map $\varphi:G(\Z)\to G(\C)$ is not currently known and depends on functorial properties of Tits' group ([Ti]).

Our notion of a `Kac--Moody group' is  the infinite dimensional analog of an {\it elementary Chevalley group} over the ring $R$, namely, a group generated by real root group generators. The notion of {\it  Chevalley group} over $R$, that is, the infinite dimensional analog of the Chevalley--Demazure group scheme, has not yet been fully formulated, though the groundwork for establishing it was given in [Ti] and discussed further in [A] and [C].

The methods for constructing Kac--Moody groups over fields  are numerous, though we will only refer to constructions of [CG] (following [G]) and [RR]).  When $R$ is a field,  Garland constructed affine Kac--Moody groups as central extensions of loop groups and characterized the dependence of a completion of $G^{\lambda}(R)$  on $\lambda$  in terms of the Steinberg cocycle ([G]).

Over fields, the group constructed in [RR] coincides with the Kac--Moody group $G(R)$ of Tits ([Ti]). Let 
$K^{\lambda}(R)=Ker(\rho_{\lambda}: G(R)\to G^{\lambda}(R))$. There are a number of known results which determine the natural extension of $K^{\lambda}$ to completions of $G(R)$ and  $G^{\lambda}(R)$ when $R$ is a field.

For symmetrizable Kac--Moody groups over fields, [BR] showed that $K^{\lambda}$ is contained in the center of a completion of $G(R)$. Recent work of Rousseau ([Rou]) shows that over fields of characteristic zero, completions $G(R)$ (as in [RR]) and $G^{\lambda}(R)$ (as in [CG])  are isomorphic as topological groups (see also [Mar]).  The complete Kac--Moody group of [CG] is constructed with respect to a choice of dominant integral weight $\lambda$. Thus Rousseau's work also shows that over fields of characteristic zero, the complete Kac--Moody groups of [CG] are independent of the choice of $\lambda$ up to isomorphism, as topological groups.

Here our interest lies in  {\it minimal} or incomplete Kac--Moody groups over rings, where the methods for determining $K^{\lambda}$ are less transparent. Thus we work only over $R=\Z$ and study the kernel of $\rho_{\lam,\Z}$  when $G$ is simply laced and hyperbolic. We hope to eventually have the techniques to extend these results to a wider class of Kac--Moody groups over more general commutative rings.

The authors wish to thank Daniel Allcock for his comments which helped to clarify the results in the last section.

\bigskip\section{Tits' Kac--Moody group}

Let $\mathfrak{g}$ be a Kac--Moody algebra with Cartan subalgebra $\mathfrak{h}$ and root space decomposition:
$$\mathfrak{g}=\mathfrak{g}^+\ \oplus\ \mathfrak{h}\ \oplus\ \mathfrak{g}^-$$
The roots $\Delta\subset\ \mathfrak{h}^*$ are the eigenvalues of the simultaneous adjoint action of $\mathfrak{h}$ on $\mathfrak{g}$, and $\mathfrak{g}^{\pm}=\oplus_{\alpha\in\Delta^{\pm}} \mathfrak{g}_{\alpha}$
where the root spaces 
$$\mathfrak{g}_{\alpha}\ =\ \{x\in \mathfrak{g}\mid[h,x]=\alpha(h)x,\ h\in \mathfrak{h}\}$$
are the corresponding eigenspaces. When $\mathfrak{g}$ is infinite dimensional, $|\Delta|=\infty$.  In this case, 
$\mathfrak{g}$ has 2 types of roots: {\it real roots} with positive norm and {\it imaginary roots} with negative or zero norm. The imaginary roots will not play a role in what follows, as we will work with Tits' presentation for Kac--Moody groups, which uses only the real root groups as generators.

Tits showed how Kac--Moody groups can be presented by generators and relations, generalizing the Steinberg presentation for finite dimensional Lie groups ([Ti]).

In the finite dimensional case, there is a  Chevalley type commutation relation of the form
$$[\chi_{\alpha}(u),\chi_{\beta}(v)]
	=
	\prod_{m,n}
	\chi_{m\alpha+n\beta}(C_{mn\alpha\beta}u^m v^n)
$$
between every pair of root groups $U_{\alpha}$, $U_{\beta}$. Here $u,v\in R$, $C_{mn\alpha\beta}$ are integers (structure constants) and the $\chi_{\alpha}$ are viewed as formal symbols in
$$U_{\alpha}=\{\chi_{\alpha}(u)\mid\alpha\in \Delta,\ u\in R\}\cong (R,+).$$
However, in the infinite dimensional case, Tits' presentation of Kac--Moody groups has infinitely many Chevalley commutation relations. To describe these, we give the following definition.

Let $(\alpha,\beta)$ be a pair of real roots and let $W$ be the Weyl group. Then $(\alpha,\beta)$ is called a \it{prenilpotent pair, }\rm if  there exist $w,\ w'\in W$ such that
{ $$w\alpha, \ w\beta\in\Delta^{re}_+{\text{ and }}w'\alpha, \ w'\beta\in\Delta^{re}_-.$$}
A pair of roots $\{\alpha,\beta\}$ is prenilpotent if and only if $\alpha\neq -\beta$ and 
{$$(\mathbb{Z}_{>0} \alpha + \mathbb{Z}_{>0} \beta )\cap \Delta^{re}_+$$} is a finite set. For every prenilpotent pair of roots $\{\alpha,\beta\}$, Tits defined  the Chevalley commutation relation
{\[
	[\chi_{\alpha}(u),\chi_{\beta}(v)]
	=
	\prod_{{{m\alpha+n\beta\in
	(\mathbb{Z}_{>0} \alpha + \mathbb{Z}_{>0} \beta )\cap \Delta^{re}_+}}}
	\chi_{m\alpha+n\beta}(C_{mn\alpha\beta}u^m v^n)
\]}
where $u,v\in R$  and $C_{mn\alpha\beta}$ are integers (structure constants).

\section{Tits' presentation}

Tits defined the {\it Steinberg group} 

$$\mathfrak{St}(R)=\ast_{\alpha\in\Delta^{re}} \ U_{\alpha}/ \text{{\tiny{Chev. comm. relns. on prenilpotent pairs}}}$$
Tits' Kac--Moody group ${G}(R)$ is a quotient of the Steinberg group $\mathfrak{St}(R)$ by some additional relations which are easy to describe.

Tits' Kac--Moody group has an infinite set of generators and an infinite set of defining relations. The sets of generators and relations are   redundant and can be reduced significantly, as shown in [A] and [AC].

In [A], Allcock defined a new functor, the {\it pre--Steinberg group} $\mathfrak{PSt}(R)$ and showed that in many cases $\mathfrak{PSt}(R)\cong \mathfrak{St}(R)$. The presentation of $\mathfrak{PSt}(R)$ is defined in terms of the Dynkin diagram rather than the full infinite root system.

\section{Simply laced hyperbolic type}

\medskip
 Carbone and Allcock obtained a further simplification of the isomorphism $\mathfrak{PSt}(R)\cong \mathfrak{St}(R)$ for root systems that are simply laced and hyperbolic ([AC]).

\medskip
{ A Dynkin diagram  is {\it simply laced} if it consists only of single bonds between nodes.}

\medskip
{ It is {\it  hyperbolic} if it is neither of affine nor or finite dimensional type, but every proper connected subdiadram is either of affine  or finite dimensional type.
}

 \begin{table}[h]
\setlength\unitlength{1mm}
\begin{picture}(110,99)(0,-4)
\put(0,2){\makebox[0pt][l]{rank 10}}
\put(15,0){\makebox[0pt][l]{
\begin{tikzpicture}[line width=\mywidth, scale=\myscale]
\draw (0,0) -- (7,0);
\draw (1,0) -- (1,1);
\draw (5,0) -- (5,1);
\vertex (0,0);
\vertex (1,0);
\vertex (2,0);
\vertex (3,0);
\vertex (4,0);
\vertex (5,0);
\vertex (6,0);
\vertex (7,0);
\vertex (1,1);
\vertex (5,1);
\end{tikzpicture}%
}}%
\put(55,0){\makebox[0pt][l]{
\begin{tikzpicture}[line width=\mywidth, scale=\myscale]
\draw (0,0) -- (8,0);
\draw (2,0) -- (2,1);
\vertex (0,0);
\vertex (1,0);
\vertex (2,0);
\vertex (3,0);
\vertex (4,0);
\vertex (5,0);
\vertex (6,0);
\vertex (7,0);
\vertex (8,0);
\vertex (2,1);
\end{tikzpicture}%
}}%
\put(0,15){\makebox[0pt][l]{rank 9}}
\put(15,13){\makebox[0pt][l]{
\begin{tikzpicture}[line width=\mywidth, scale=\myscale]
\draw (0,0) -- (6,0);
\draw (1,0) -- (1,1);
\draw (4,0) -- (4,1);
\vertex (0,0);
\vertex (1,0);
\vertex (2,0);
\vertex (3,0);
\vertex (4,0);
\vertex (5,0);
\vertex (6,0);
\vertex (1,1);
\vertex (4,1);
\end{tikzpicture}%
}}%
\put(50,13){\makebox[0pt][l]{
\begin{tikzpicture}[line width=\mywidth, scale=\myscale]
\draw (0,0) -- (7,0);
\draw (3,0) -- (3,1);
\vertex (0,0);
\vertex (1,0);
\vertex (2,0);
\vertex (3,0);
\vertex (4,0);
\vertex (5,0);
\vertex (6,0);
\vertex (7,0);
\vertex (3,1);
\end{tikzpicture}%
}}%
\put(90,9){\makebox[0pt][l]{
\begin{tikzpicture}[line width=\mywidth, scale=\myscale]
\draw 
(180:2.307) -- 
(180:1.307) -- 
(135:1.307) -- 
(90:1.307) -- 
(45:1.307) --
(0:1.307) --
(-45:1.307) --
(-90:1.307) -- 
(-135:1.307) --
(180:1.307);
\vertex (180:2.307); 
\vertex (180:1.307); 
\vertex (135:1.307); 
\vertex (90:1.307); 
\vertex (45:1.307);
\vertex (0:1.307);
\vertex (-45:1.307);
\vertex (-90:1.307); 
\vertex (-135:1.307); 
\end{tikzpicture}%
}}%
\put(0,28){\makebox[0pt][l]{rank 8}}
\put(15,27){\makebox[0pt][l]{
\begin{tikzpicture}[line width=\mywidth, scale=\myscale]
\draw (0,0) -- (5,0);
\draw (1,0) -- (1,1);
\draw (3,0) -- (3,1);
\vertex (0,0);
\vertex (1,0);
\vertex (2,0);
\vertex (3,0);
\vertex (4,0);
\vertex (5,0);
\vertex (1,1);
\vertex (3,1);
\end{tikzpicture}%
}}%
\put(45,27){\makebox[0pt][l]{
\begin{tikzpicture}[line width=\mywidth, scale=\myscale]
\draw (0,0) -- (5,0);
\draw (2,0) -- (2,2);
\vertex (0,0);
\vertex (1,0);
\vertex (2,0);
\vertex (3,0);
\vertex (4,0);
\vertex (5,0);
\vertex (2,1);
\vertex (2,2);
\end{tikzpicture}%
}}%
%
\put(75,26){\makebox[0pt][l]{
\begin{tikzpicture}[line width=\mywidth, scale=\myscale]
\draw 
(180:2.152) -- 
(180:1.152) -- 
(128:1.152) -- 
(77:1.152) -- 
(26:1.152) --
(-26:1.152) --
(-77:1.152) --
(-128:1.152) --
(180:1.152);
\vertex (180:2.152);
\vertex (180:1.152);
\vertex (128:1.152);
\vertex (77:1.152);
\vertex (26:1.152);
\vertex (-26:1.152);
\vertex (-77:1.152);
\vertex (-128:1.152);
\vertex (180:1.152);
\end{tikzpicture}%
}}%
\put(0,46){\makebox[0pt][l]{rank 7}}
\put(15,45){\makebox[0pt][l]{
\begin{tikzpicture}[line width=\mywidth, scale=\myscale]
\draw (0,0) -- (4,0);
\draw (1,0) -- (1,1);
\draw (2,0) -- (2,1);
\vertex (0,0);
\vertex (1,0);
\vertex (2,0);
\vertex (3,0);
\vertex (4,0);
\vertex (1,1);
\vertex (2,1);
\end{tikzpicture}%
}}%
\put(40,43){\makebox[0pt][l]{
\begin{tikzpicture}[line width=\mywidth, scale=\myscale]
\draw (-2,0)--(-1,0)--(-.5,.866)--(.5,.866)--(1,0)--(.5,-.866)--(-.5,-.866)--(-1,0);
\vertex (-2,0);
\vertex (-1,0);
\vertex (-.5,.866);
\vertex (.5,.866);
\vertex (1,0);
\vertex (.5,-.866);
\vertex (-.5,-.866);
\end{tikzpicture}%
}}%
\put(0,61.5){\makebox[0pt][l]{rank 6}}
\put(15,57){\makebox[0pt][l]{
\begin{tikzpicture}[line width=\mywidth, scale=\myscale]
\draw (0,0) -- (3,0);
\draw (2,-1) -- (2,1);
\vertex (0,0);
\vertex (1,0);
\vertex (2,0);
\vertex (3,0);
\vertex (2,1);
\vertex (2,-1);
\end{tikzpicture}%
}}%
\put(35,57.3){\makebox{
\begin{tikzpicture}[line width=\mywidth, scale=\myscale]
\draw (0,0) -- (-1,0);
\draw (0,0) -- (-.309,.951);
\draw (0,0) -- (-.309,-.951);
\draw (0,0) -- (.809,.588);
\draw (0,0) -- (.809,-.588);
\vertex (0,0);
\vertex (-1,0);
\vertex (-.309,.951);
\vertex (-.309,-.951);
\vertex (.809,.588);
\vertex (.809,-.588);
\end{tikzpicture}%
}}%
\put(49,58){\makebox[0pt][l]{
\begin{tikzpicture}[line width=\mywidth, scale=\myscale]
\draw (-1.851,0)--(-.851,0)--(-.263,.809)--(.688,.5)--(.688,-.5)--(-.263,-.809)--(-.851,0);
\vertex (-1.851,0);
\vertex (-.851,0);
\vertex (-.263,.809);
\vertex (-.263,-.809);
\vertex (.688,.5);
\vertex (.688,-.5);
\end{tikzpicture}%
}}%
\put(0,76){\makebox[0pt][l]{rank 5}}
\put(15,73){\makebox[0pt][l]{
\begin{tikzpicture}[line width=\mywidth, scale=\myscale]
\draw (0,0) -- (0,1.414) -- (1.414,1.414) -- (1.414,0) -- (0,0);
\draw (0,0) -- (1.414,1.414);
\vertex (0,0);
\vertex (0,1.414);
\vertex (1.414,0);
\vertex (1.414,1.414);
\vertex (.707,.707);
\end{tikzpicture}%
}}%
\put(28,73){\makebox[0pt][l]{
\begin{tikzpicture}[line width=\mywidth, scale=\myscale]
\draw (0,0) -- (.707,.707) -- (1.414,0) -- (.707,-.707) -- (0,0);
\draw (0,0) -- (-1,0);
\vertex (0,0);
\vertex (.707,.707);
\vertex (.707,-.707);
\vertex (1.414,0);
\vertex (-1,0);
\end{tikzpicture}%
}}%
\put(0,89){\makebox[0pt][l]{rank 4}}
\put(15,86){\makebox[0pt][l]{
\begin{tikzpicture}[line width=\mywidth, scale=\myscale]
\draw (0,1) -- (-.866,-.5) -- (.866,-.5) -- (0,1);
\draw (0,0) -- (0,1);
\draw (0,0) -- (-.866,-.5);
\draw (0,0) -- (.866,-.5);
\vertex (0,0);
\vertex (0,1);
\vertex (-.866,-.5);
\vertex (.866,-.5);
\end{tikzpicture}%
}}%
\put(30,87){\makebox{
\begin{tikzpicture}[line width=\mywidth, scale=\myscale]
\draw (0,0) -- (0,1) -- (1,1) -- (1,0) -- (0,0);
\draw (0,0) -- (1,1);
\vertex (0,0);
\vertex (0,1);
\vertex (1,0);
\vertex (1,1);
\end{tikzpicture}%
}}%
\put(40,87){\makebox[0pt][l]{
\begin{tikzpicture}[line width=\mywidth, scale=\myscale]
\draw (0,0) -- (.732,.5) -- (.732,-.5) -- (0,0);
\draw (0,0) -- (-1,0);
\vertex (0,0);
\vertex (.732,.5);
\vertex (.732,-.5);
\vertex (-1,0);
\end{tikzpicture}%
}}%
\end{picture}
\caption{The simply-laced hyperbolic Dynkin diagrams.  The {\it
  rank\/} equals the number of nodes.}
\label{tab-diagrams}
\end{table}

\newpage

\bigskip
\section{Finitely many defining relations parametrized over $R$}
{}{{The Kac--Moody group $G(R)$ is generated by elements $X_i(t)$, $t\in R$, $S_i$, $i\in \{1,\dots,\ell\}$. The elements $\widetilde{h}_i(a)$  in the first column of  Table 1 are}}

{{{{$$\widetilde{h}_i(a):=\widetilde{s}_i(a)\widetilde{s}_i(-1),$$
$$\widetilde{s}_i(a):=X_i(a)S_iX_i(a^{-1})S_i^{-1}X_i(a).$$}}}}

\begin{table}[h]
{\tiny{
\begin{tabular}{| l | l |  l  | l |}
\hline
&              &                &  \\
Any $i$, all $a,b\in R^{\times}$ & Each $i$, $t,u\in R$ &$i\neq j$ not adjacent  & $i\neq j$ adjacent  \\
&  & $t,u\in R$  &  $t,u\in R$  \\

&              &                &  \\
\hline
&              &                &  \\

$\widetilde{h}_i(a)\widetilde{h}_i(b)=\widetilde{h}_i(ab)$ &   $X_i(t)X_i(u)=X_i(t+u)$   &      $S_iS_j=S_jS_i$            &  $S_iS_jS_i=S_jS_iS_j$  \\

&              &                &   $S_i^2S_jS_i^{-2}=S_j^{-1}$  \\

      &   &                            &                                      \\
 &   $[S_i^2,X_i(t)]=1$  &    $[S_i,X_j(t)]=1   $&    $X_i(t) S_j S_i=S_j S_i X_j(t)$\\

&              &                &   $S_i^2 X_j(t) S_i^{-2}=X_j(t)^{-1}$\\

      &   &                            &                                      \\

    &  $S_i=X_i(1) S_i X_i(1) S_i^{-1} X_i(1)$              &   $[X_i(t),X_j(u)]=1   $&  $[X_i(t),X_j(u)]=S_i X_j(tu) S_i^{-1}$\\

      &   &                            &                                      \\

  &            &                &  $[X_i(t),S_i X_j(u) S_i^{-1}]=1$\\

\hline

\end{tabular} }}

\caption{{{The defining relations for $G(R)$, $G$ simply laced and hyperbolic, $R$ a commutative ring with 1}}}

\end{table}

Over $R=\mathbb{Z}$, the generators $X_i(u)$ are obtained from $X_i=X_i(1)$ via $X_i(u)=X_i^u$.

Thus we can rewrite the  presentation without the scalars from the underlying ring and obtain a {\it finite} presentation for $G(\mathbb{Z})$ as in Table 2.

\begin{table}[h]
{\small{
\begin{tabular}{| l | l | l | l |}
\hline
& & & \\
Any $i$ & Each $i=\{1,\dots,\ell\}$ &$i\neq j$ not adjacent & $i\neq j$ adjacent \\
& & & \\

\hline

$ (1) \ S_i^4=1$ & $ (2) \ [S_i^2,X_i]=1$ & $ (4) \ S_iS_j=S_jS_i$ & $(7) \  S_iS_jS_i=S_jS_iS_j$ \\

& & & $(8) \ S_i^2S_jS_i^{-2}=S_j^{-1}$ \\

& $ (3) \ S_i=X_i S_i X_i S_i^{-1} X_i$ & $(5) \ [S_i,X_j]=1 $& $(9) \ X_i S_j S_i=S_j S_i X_j$\\

& & & $(10) \ S_i^2 X_j S_i^{-2}=X_j^{-1}$\\

& & $(6) \ [X_i,X_j]=1 $& $(11) \ [X_i,X_j]=S_i X_j S_i^{-1}$\\

& & & $(12)  \ [X_i,S_i X_j S_i^{-1}]=1$\\

\hline

\end{tabular} }}

\caption{The defining relations for $G(\mathbb{Z})$, $G$ simply laced and hyperbolic}

\end{table}


\bigskip\section{Representation--theoretic Kac--Moody groups over rings}

Here we describe the construction of a representation theoretic  Kac--Moody group $G^{\lambda}(R)$, over any commutative ring $R$ with 1, constructed using an integrable highest weight module $V^{\lambda}$ for any symmetrizable Kac--Moody algebra and a $\Z$--form ${\mathcal U}_{{\mathbb{Z}}}$ of the universal enveloping algebra ${\mathcal U}_{R}$. This construction was developed in [Ca] following the methods of [CG] and is a natural generalization of the theory of  elementary Chevalley  groups over commutative rings (see for example [VP]).

Let $\mathfrak{g}=\mathfrak{g}(A)$ be Kac--Moody algebra  corresponding to a symmetrizable generalized Cartan matrix $A=(a_{ij})_{i,j\in I}$.
Let $V^{\lambda}$ be the unique irreducible highest weight module for $\mathfrak{g}$ corresponding to dominant integral weight $\lambda$.

We let $\Lambda\subseteq \mathfrak{h}^*$ be the linear span of  the simple roots $\alpha_i$, for $i\in I$, and let $\Lambda^{\vee}\subseteq
\mathfrak{h}$ be the linear span of the simple coroots
$\alpha^{\vee}_i$, for $i\in I$.

Let $e_i$ and $f_i$ be the Chevalley generators of $\mathfrak{g}$. Let ${\mathcal U}_{\mathbb{C}}$ be the universal enveloping algebra of $\mathfrak{g}$. Let

\medskip\noindent ${\mathcal U}_{{\mathbb{Z}}}\subseteq {\mathcal U}_{{\mathbb{C}}}$ be the ${\mathbb{Z}}$--subalgebra generated by $\dfrac{e_i^{m}}{m!}$,
$\dfrac{f_i^{m}}{m!}$ for $i\in I$ and $\left (\begin{matrix}
h\\ m\end{matrix}\right )$, for
$h\in\Lambda^{\vee}$ and
$m\geq 0$,

\medskip\noindent ${\mathcal U}^+_{{\mathbb{Z}}}$ be the ${\mathbb{Z}}$-subalgebra generated by $\dfrac{e_i^{m}}{m!}$ for $i\in I$
and
$m\geq 0$,

\medskip\noindent ${\mathcal U}^-_{{\mathbb{Z}}}$ be the ${\mathbb{Z}}$--subalgebra generated by $\dfrac{f_i^{m}}{m!}$ for $i\in I$
and
$m\geq 0$,  

\medskip\noindent ${\mathcal U}^0_{{\mathbb{Z}}}\subseteq {\mathcal U}_{{\mathbb{C}}}(\mathfrak{h})$ be the ${\mathbb{Z}}$--subalgebra generated by $\left (\begin{matrix}
h \\ m\end{matrix}\right )$, for
$h\in\Lambda^{\vee}$ and $m\geq 0$.   We set


 $$\mathfrak{g}_{\Z}\ =\ \mathfrak{g}_{\C}\cap {\mathcal U}_{\Z},$$ 
$$\mathfrak{g}^{\pm}_{{\Z}}=\mathfrak{g}^{\pm}_{{\C}}\cap {\mathcal U}_{\Z},$$ 
 $$\mathfrak{h}_{\Z}\ =\ \mathfrak{h}_{\C}\cap {\mathcal U}_{\Z}.$$ 
Then $\mathfrak{g}_{\C}=\mathfrak{g}_{\Z}\otimes_{\Z}\C$  and $\mathfrak{h}_{\Z}$ is generated by $\left (\begin{matrix}
h \\ m\end{matrix}\right )$, for
$h\in\Lambda^{\vee}$ and $m\geq 0$.

For $R$ a commutative ring with 1, set  $\mathfrak{g}_{R}=\mathfrak{g}_{\Z}\otimes_{\Z} R$. Then $\mathfrak{g}_{R}$ is the infinite dimensional analog of the {\it Chevalley algebra} over $R$. Set
$\mathfrak{h}_{R}=\mathfrak{h}_{\Z}\otimes R$ and for $\a\in\Delta$, let 
$\mathfrak{g}^{\alpha}_{R} =\mathfrak{g}^{\alpha}_{\Z} \otimes R$. Then
$$\mathfrak{g}_{R}\ =\ \mathfrak{n}_{R}^-\oplus \mathfrak{h}_{R} \oplus \mathfrak{n}_{R}^+$$
is the root space decomposition of $\mathfrak{g}_{R}$ relative to $\mathfrak{h}_{R}$, where 
$$\mathfrak{n}_{R}^-=\bigoplus_{\a\in\Delta^-}\mathfrak{n}_{R}^{\a},\ \mathfrak{n}_{R}^+=\bigoplus_{\a\in\Delta^+}\mathfrak{n}_{R}^{\a}.$$

\medskip \noindent We now consider the orbit of our highest weight vector $v_{\lambda}\in V$ under $\mathcal{U}_{\mathbb{Z}}$. We have
$$\mathcal{U}^+_{\mathbb{Z}}v_{\lambda}=\mathbb{Z} v_{\lambda}$$
since all elements of ${\mathcal U}^+_{{\mathbb{Z}}}$ except for 1 annihilate $v_{\lambda}$. Also
$$\mathcal{U}^0_{\mathbb{Z}}v_{\lambda}=\mathbb{Z} v_{\lambda}$$
since ${\mathcal U}^0_{{\mathbb{Z}}}$ acts as scalar multiplication on $v_{\lambda}$ by a $\mathbb{Z}$--valued scalar. 

Thus
$$\mathcal{U}_{\mathbb{Z}}\cdot v_{\lambda}=\mathcal{U}^-_{\mathbb{Z}}\cdot (\mathbb{Z} v_{\lambda})=\mathcal{U}^-_{\mathbb{Z}}\cdot(v_{\lambda}).$$

Define $\mathcal{U}(\mathfrak{g})_{R}=\mathcal{U}(\mathfrak{g})_{\Z}\otimes R$. This is the infinite dimensional analog of a {\it hyperalgebra} as in [C] (see also [Hu]). We set 
$$V^{\lambda}_{\mathbb{Z}}\  =\ \mathcal{U}_{\mathbb{Z}}\cdot v_{\lambda}\ =\ \mathcal{U}^-_{\mathbb{Z}}\cdot(v_{\lambda})$$
Then $V^{\lambda}_{\mathbb{Z}}$  is a lattice in $V^{\lambda}_{R}=R\otimes_{\mathbb{Z}}V^{\lambda}_{\mathbb{Z}}$ and a ${\mathcal U}_\mathbb{Z}$--module.
 For each weight $\mu$ of  $V$, let  $V^{\lambda}_{\mu}$ be the corresponding weight
space.  Set
$$V^{\lambda}_{R}\ =\  {R}\otimes_{\mathbb{Z}} V^{\lambda}_{\mathbb{Z}}.$$
and
$$V^{\lambda}_{\mu,R}\ =\ R\otimes_{\Z} V^{\lambda}_{\mu,\Z}$$
so that 
$$V^{\lambda}_{R}\  =\ \oplus_{\mu}V^{\lambda}_{\mu,{R}}.$$
Here, $V_R=V_{\Z}\otimes_{\Z} R$ free $R$--module with basis $v_{\lambda}\otimes 1$ and a module over $\mathfrak{g}_R$ called the {\it Weyl module}. For $s,t\in R$ and $i\in I$,  set 
$$\chi_{\alpha_i}(t)\ =\ exp(\rho(se_i)),$$
$$\chi_{-\alpha_i}(t)\ =\ exp(\rho(tf_i)),$$
where $\rho$ is the defining representation for $V$, and we set
$$\widetilde{w}_{\alpha_i}= \chi_{\alpha_i}(t)\chi_{-\alpha_i}(-t^{-1})\chi_{\alpha_i}(t),$$
$$h_{\alpha_i}(t)\ =\ \widetilde{w}_{\alpha_i}(t) \widetilde{w}_{\alpha_i}(1)^{-1}.$$
Then these are elements of $Aut(V^{\lambda}{R})$, thanks to the
local
nilpotence of $e_{i},$ $f_{i}.$

The group $\widetilde{W}$ is known as the {\it extended Weyl group}. We will use a presentation of $\widetilde{W}$ given in [KP].
 Let $H^{\lambda}\leq G_{\lambda}$ be the subgroup generated by the elements $h_{\alpha_i}(t)$, $t\in R^{\times}$, $i\in I$.

We let $G^{\lambda}(R)\leq Aut(V^{\lambda}_{R})$ be the group: 
$$G^{\lambda}(R)=\langle \chi_{\alpha_i}(s),\ \chi_{-\alpha_i}(t)\mid i\in I,\ s,t\in R\rangle.$$ 
We may refer to $G^{\lambda}(R)$ as a `representation--theoretic Kac--Moody group'. We summarize the construction in the following.

\begin{theorem} Let $\mathfrak{g}$ be a symmetrizable Kac--Moody algebra over a commutative ring $R$ with 1. Let  $\alpha_i$, $i\in I$, be the simple roots and $e_i$, $f_i$ the Chevalley generators of $\mathfrak{g}$. Let $V^{\lambda}_R$ be an $R$--form of an integrable highest weight module $V^{\lambda}$ for $\mathfrak{g}$, corresponding to  dominant integral weight $\lambda$ and defining representation $\rho:\mathfrak{g}\to End(V^{\lambda}_R)$. Then 
$$G^{\lambda}(R)=\langle  \chi_{\a_i}(s)=exp(\rho(se_i)),\ \chi_{-\a_i}(t)=exp(\rho(tf_i))\mid s,t\in R\rangle\leq Aut(V^{\lambda}_R)$$
is a representation--theoretic  Kac--Moody group associated to $\mathfrak{g}$.
\end{theorem}

A similar construction for $G^{\lambda}$ over arbitrary fields was used in [CG] to construct Kac--Moody groups over finite fields.

When $R=\C$ (or $R=\Q$), we define the integral subgroup $G^{\lambda}(\Z)$ to be the group 
$$G^{\lambda}(\Z)=\langle\chi_{\a_i}(s),\chi_{-\a_i}(t)\mid s,t\in\Z,\ i\in I\rangle.$$
It is non--trivial to prove that this group coincides with the `Chevalley group' over $\Z$, namely the subgroup of  $G^{\lambda}(\C)$ preserving $V^{\lambda}(\Z)$. This is proven in [Ca].  The following was also proven in [Ca]. 

\begin{theorem} ([Ca]) As a subgroup of $G^{\lambda}(\C)$ (or  $G^{\lambda}(\Q)$) the group $G^{\lambda}(\Z)$ has the following  generating sets:

(1) $\chi_{\alpha_i}(1)$ and $\chi_{-\alpha_i}(1)$, 

and

(2) $\chi_{\alpha_i}(1)$ and $\widetilde{w}_{\alpha_i}(1)=\chi_{\alpha_i}(1)\chi_{-\alpha_i}(-1)\chi_{\alpha_i}(1)$.
\end{theorem}

We now take  $\lam$ to be a regular weight, and 
we define the \emph{weight topology} on $G^{\lam}(R)$
by taking stabilizers of elements of $V^{\lam}_{R}$ as a sub--base of
neighborhoods of the identity. The completion of $G^{\lam}(R)$
in this topology will be referred to as the \emph{Carbone--Garland completion}
and denoted by $\Ghat^{\lam}(A)$. Since $G^{\lam}(R)$ is a homomorphic image of
$G(R)$, we can think of $\Ghat^{\lam}(R)$ as a completion of $G(R)$ rather than $G^{\lam}(R)$.


\section{Uniqueness of representation--theoretic  Kac--Moody groups over $\Z$}

For $G$ simply laced and hyperbolic, we consider a map from the finite presentation of $G(R)$ of [AC]  to the representation--theoretic group $G^{\lambda}(R)$,  over a commutative ring $R$ with 1.

We define the map $\rho_{\lambda}$:

\vspace{-0.4cm}
{
\[
	\begin{CD}
	G(R)  @> \rho_{\lambda} >>   G^\lambda(R)
	\end{CD}
\]}

from generators of $G(R)$ to generators of $G^\lambda(R)$:

 {$$\quad\qquad  X_i(t)     \mapsto  \chi_{\alpha_i}(t)  $$
$$\quad\qquad    S_i   \mapsto  \widetilde{w}_{\alpha_i}$$
$$\quad\qquad \widetilde{h}_i(u)    \mapsto h_{\alpha_i}(u)$$}
for $t\in{R}$, $u\in R^{\times}$, $i\in\{1,\dots ,\ell\}$. 

\bigskip
\begin{theorem} When $R=\Z$, this map extends to a group homomorphism
$$\rho_{\lambda,\Z}: G(\Z)  \to   G^{\lambda}(\Z).$$
\end{theorem}

\medskip
{\it Remark.} The existence of an abstract group homomorphism from Tits' group $G(\Q)\to G^{\lambda}(\Q)$ and its restriction $G(\Z)\to G^{\lambda}(\Z)$ can be deduced from [Ti] and [CER]. However, we are interested only in this homomorphism from the finite presentation of $G(\Z)$ in [AC] and in determining the generators of its kernel. Thus we prove that the relations in the finite presentation of $G(\Z)$ are satisfied in $G^{\lambda}(\Z)$.

\medskip
\begin{proof} The presentation in [AC] for $G(\Z)$ does not make signs of structure constants explicit, but rather holds formally for any choice of sign of the structure constants. The known relations in $G^{\lambda}(\Z)$ also hold up to a sign in the structure constants.  Hence we verify that $\rho_{\lambda,\Z}$ is a homomorphism  up to a sign in the structure constants. From now on, we will write $\rho_{\lambda}$ to denote $\rho_{\lambda,\Z}$ if the context allows.

We  show that the relations in $G(\Z)$ are satisfied by their images in $G^{\lambda}(\Z)$ under $\rho_{\lambda}$. We refer to the labelling of the relations in Table 2.

\medskip
(1) Consider the image of $S_i^4=1$ under $\rho_{\lambda}$.  For any $i$ we have:
$$\rho_\lambda(S_i^4)=\widetilde{w}_{\alpha_i}^4.$$
Since the $\widetilde{w}_{\alpha_i}$ generate the extended Weyl group $\widetilde{W}$  and $\widetilde{w}_{\alpha_i}^4=1$ in $\widetilde{W}$, it follows that $$\rho_\lambda(S_i^4)=\widetilde{w}_{\alpha_i}^4=1.$$ Thus, the image of the relation $S_i^4=1$ is satisfied in $G^\lambda(\Z)$.

\medskip
(2) In $G^\lambda(\Z)$, we have
$$\widetilde{w}_{\alpha_i}\chi_{\a_j}\widetilde{w}_{\alpha_i}^{-1}=\chi_{w_{\alpha_i}(\a_j)}$$ for any pair $i,j$.
For each $i$ we have $w_{\a_i}(\a_i)=-\a_i.$
Hence
 $$\widetilde{w}_{\alpha_i}\chi_{\a_i}\widetilde{w}_{\alpha_i}^{-1}=\chi_{-\a_i}.$$
Consider now the image of $[S_i^2,X_i]$ in $G^\lambda(\Z)$:
$$\rho_\lambda([S_i^2,X_i])=\rho_\lambda(S_iS_iX_iS_i^{-1}S_i^{-1}X_i^{-1})=\widetilde{w}_{\a_i}\widetilde{w}_{\a_i}\chi_{\a_i}\widetilde{w}_{\a_i}^{-1}\widetilde{w}_{\a_i}^{-1}\chi_{\a_i}^{-1}.$$
This reduces to $$\rho_\lambda([S_i^2,X_i])=\widetilde{w}_{\a_i}(\widetilde{w}_{\a_i}\chi_{\a_i}\widetilde{w}_{\a_i}^{-1})\widetilde{w}_{\a_i}^{-1}\chi_{\a_i}^{-1}=\widetilde{w}_{\a_i}\chi_{-\a_i}\widetilde{w}_{\a_i}^{-1}\chi_{\a_i}^{-1}=\chi_{\a_i}\chi_{\a_i}^{-1}=1.$$
Thus the image of the relation $[S_i^2,X_i]=1$ is satisfied in $G^\lambda(\Z)$.

\medskip
(3) Similarly, consider the image of  $X_iS_iX_iS_i^{-1}X_i$. Applying $\rho_{\lambda}$, we have
$$\rho_\lambda(X_iS_iX_iS_i^{-1}X_i)=\chi_{\a_i}\widetilde{w}_{\a_i}\chi_{\a_i}\widetilde{w}_{\a_i}^{-1}\chi_{\a_i}=\chi_{\a_i}\chi_{-\a_i}\chi_{\a_i}.$$
This last term equals $\widetilde{w}_{\a_i}$, so that $$\rho_\lambda(X_iS_iX_iS_i^{-1}X_i)=\widetilde{w}_{\a_i}=\rho_\lambda(S_i).$$
Thus, the image of the relation $S_i=X_iS_iX_iS_i^{-1}X_i$ is satisfied in $G^\lambda(\Z)$.

\medskip
(4) Suppose now that $i\neq j$ are not adjacent on the Dynkin diagram. Using the presentation of  $\widetilde{W}$ in [KP], we have $$\widetilde{w}_{\a_i}\widetilde{w}_{\a_j}\widetilde{w}_{\a_i}^{-1}\widetilde{w}_{\a_j}^{-1}=1.$$
Hence, $\widetilde{w}_{\a_i}\widetilde{w}_{\a_j}=\widetilde{w}_{\a_j}\widetilde{w}_{\a_i}$, so that $\rho_\lambda(S_iS_j)=\rho_\lambda(S_jS_i)$.
Thus the image of the relation $S_iS_j=S_jS_i$ for $i\neq j$, $i,j$  not adjacent on the Dynkin diagram, is satisfied in $G^\lambda(\Z)$.

\medskip
(5) Consider the image of $[S_i,X_j]$ under $\rho_\lambda$: $$\rho_\lambda([S_i,X_j])=\rho_\lambda(S_iX_jS_i^{-1}X_j^{-1})=\widetilde{w}_{\a_i}\chi_{\a_j}\widetilde{w}_{\a_i}^{-1}\chi_{\a_j}^{-1}=\chi_{w_{\a_i}(\a_j)}\chi_{\a_j}^{-1}.$$
In this last term, we evaluate $w_{\a_i}(\a_j)$: $$w_{\a_i}(\a_j)=\a_j-\a_j(\a_i^{\vee})\a_i=\a_j-(a_{ij})\a_i.$$
However, for $i\neq j$, $i,j$ not adjacent on the Dynkin diagram, we have $a_{ij}=0$, so that $$w_{\a_i}(\a_j)=\a_j-(a_{ij})\a_i=\a_j.$$
If $i\neq j$ are not adjacent, then $$\rho_\lambda([S_i,X_j])=\chi_{\a_j}\chi_{\a_j}^{-1}=1.$$
Thus, the image of $[S_i,X_j]=1$ for $i\neq j$, $i,j$ not adjacent on the Dynkin diagram is satisfied in $G^\lambda(\Z)$.

\medskip
(6) For $i\neq j$, and $i,j$  not adjacent, we may compute in the rank 2 root subsystem of type $A_1 \times A_1$. We have  $[\chi_{\a_i},\chi_{\a_j}]=1$, so that $\rho_\lambda([X_i,X_j])=1$. Thus the image of the relation $[X_i,X_j]=1$ for $i\neq j$ not adjacent, is satisfied in $G^\lambda(\Z)$.

\medskip
(7) We now suppose that $i\neq j$ are adjacent on the Dynkin diagram. Using the presentation of $\widetilde{W}$ in [KP], we have
 $$\widetilde{w}_{\a_i}\widetilde{w}_{\a_j}\widetilde{w}_{\a_i}\widetilde{w}_{\a_j}^{-1}\widetilde{w}_{\a_i}^{-1}\widetilde{w}_{\a_j}^{-1}=1.$$
Hence, $\widetilde{w}_{\a_i}\widetilde{w}_{\a_j}\widetilde{w}_{\a_i}=\widetilde{w}_{\a_j}\widetilde{w}_{\a_i}\widetilde{w}_{\a_j}$, so that $$\rho_\lambda(S_iS_jS_i)=\widetilde{w}_{\a_i}\widetilde{w}_{\a_j}\widetilde{w}_{\a_i}=\widetilde{w}_{\a_j}\widetilde{w}_{\a_i}\widetilde{w}_{\a_j}=\rho_\lambda(S_jS_iS_j).$$
Thus the image of the relation $S_iS_jS_i=S_jS_iS_j$ for $i\neq j$, $i,j$ adjacent is satisfied in $G^\lambda(\Z)$.

\medskip
(8) We have the relation $\widetilde{w}_{\a_j}\widetilde{w}_{\a_i}^2\widetilde{w}_{\a_j}^{-1}=\widetilde{w}_{\a_i}^2(\widetilde{w}_{\a_j})^{-2a_{ij}}$ in $G^\lambda(\Z)$ (see [CER]). If $i\neq j$ are adjacent, then $a_{ij}=-1$. Hence $$\widetilde{w}_{\a_j}\widetilde{w}_{\a_i}^2\widetilde{w}_{\a_j}^{-1}=\widetilde{w}_{\a_i}^2\widetilde{w}_{\a_j}^2.$$
Multiplying by $\widetilde{w}_{\a_j}$ on both sides yields  $\widetilde{w}_{\a_j}\widetilde{w}_{\a_i}^2=\widetilde{w}_{\a_i}^2\widetilde{w}_{\a_j}^3.$ Since  $\widetilde{w}_{\a_i}^4=1$ for all $i$ in $G^\lambda({\Z})$, we can rewrite the right hand side as $\widetilde{w}_{\a_i}^2\widetilde{w}_{\a_j}^{-1}$.
We have $\widetilde{w}_{\a_i}^2=\widetilde{w}_{\a_i}^{-2}$ for all $i$, so $\widetilde{w}_{\a_j}\widetilde{w}_{\a_i}^{-2}=\widetilde{w}_{\a_i}^{-2}\widetilde{w}_{\a_j}^{-1}$. Finally, multiplying by $\widetilde{w}_{\a_i}^2$ on both sides gives $$\rho_\lambda(S_i^2S_jS_i^{-2})=\widetilde{w}_{\a_i}^2\widetilde{w}_{\a_j}\widetilde{w}_{\a_i}^{-2}=\widetilde{w}_{\a_j}^{-1}=\rho_\lambda(S_j^{-1}).$$ Thus, the image of the relation $S_i^2S_jS_i^{-2}=S_j^{-1}$ for $i,j$ adjacent is satisfied in $G^\lambda(\Z)$.

\medskip
(9) Now consider the words $\widetilde{w}_{\a_i}\chi_{\a_j}\widetilde{w}_{\a_i}^{-1}$ and $\widetilde{w}_{\a_j}^{-1}\chi_{\a_i}\widetilde{w}_{\a_j}$ in $G^\lambda(\Z)$. As before, we have
$$\widetilde{w}_{\alpha_i}\chi_{\a_j}\widetilde{w}_{\alpha_i}^{-1}=\chi_{w_{\a_i}(\a_j)}$$ for any $i,j$. Using the relation $\widetilde{w}_{\a_i}^4=1$ for all $i$, we may rewrite the second word above as $$\widetilde{w}_{\a_j}^{-1}\chi_{\a_i}\widetilde{w}_{\a_j}=\widetilde{w}_{\a_j}^3\chi_{\a_i}\widetilde{w}_{\a_j}^{-3}.$$
For $i,j$ adjacent on the Dynkin diagram, $a_{ij}=-1$, so 
$$w_{\a_j}(\a_i)=\a_i-\a_i(\a_j^{\vee})\a_j=\a_i-(a_{ij})\a_j=\a_i+\a_j.$$
Thus
$$\widetilde{w}_{\a_j}\chi_{\a_i}\widetilde{w}_{\a_j}^{-1}=\chi_{w_{\a_j}(\a_i)}=\chi_{\a_i+\a_j}.$$
Similarly, $$w_{\a_i}(\a_j)=\a_j-\a_j(\a_i^{\vee})\a_i=\a_j-(a_{ji})\a_i=\a_j+\a_i=\a_i+\a_j.$$
We may simplify the first word: $$\widetilde{w}_{\alpha_i}\chi_{\a_j}\widetilde{w}_{\alpha_i}^{-1}=\chi_{\a_i+\a_j}$$ and the second word: $$\widetilde{w}_{\a_j}^{-1}\chi_{\a_i}\widetilde{w}_{\a_j}=\widetilde{w}_{\a_j}^3\chi_{\a_i}\widetilde{w}_{\a_j}^{-3}=\widetilde{w}_{\a_j}^2(\widetilde{w}_{\a_j}\chi_{\a_i}\widetilde{w}_{\a_j})\widetilde{w}_{\a_j}^{-2}=\widetilde{w}_{\a_j}^2\chi_{\a_i+\a_j}\widetilde{w}_{\a_j}^{-2}.$$
This yields
 $$\widetilde{w}_{\a_j}^{-1}\chi_{\a_i}\widetilde{w}_{\a_j}=\widetilde{w}_{\a_j}^2\chi_{\a_i+\a_j}\widetilde{w}_{\a_j}^{-2}=\widetilde{w}_{\a_j}(\widetilde{w}_{\a_j}\chi_{\a_i+\a_j}\widetilde{w}_{\a_j}^{-1})\widetilde{w}_{\a_j}^{-1}=\widetilde{w}_{\a_j}\chi_{w_{\a_j}(\a_i+\a_j)}\widetilde{w}_{\a_j}^{-1}.$$
We have $$w_{\a_j}(\a_i+\a_j)=w_{\a_j}(\a_i)+w_{\a_j}(\a_j).$$
Since $w_{\a_j}(\a_i)=\a_i+\a_j$ and $w_{\a_j}(\a_j)=-a_j$, we have $$w_{\a_j}(\a_i+\a_j)=a_i.$$
Hence $$\widetilde{w}_{\a_j}^{-1}\chi_{\a_i}\widetilde{w}_{\a_j}=\widetilde{w}_{\a_j}\chi_{w_{\a_j}(\a_i+\a_j)}\widetilde{w}_{\a_j}^{-1}=\widetilde{w}_{\a_j}\chi_{\a_i}\widetilde{w}_{\a_j}^{-1}$$
which yields $$\widetilde{w}_{\a_j}^{-1}\chi_{\a_i}\widetilde{w}_{\a_j}=\widetilde{w}_{\a_j}\chi_{\a_i}\widetilde{w}_{\a_j}^{-1}=\chi_{w_{\a_j}(\a_i)}=\chi_{\a_i+\a_j}.$$
Hence $$\widetilde{w}_{\a_j}\chi_{\a_i}\widetilde{w}_{\a_j}^{-1}=\chi_{\a_i+\a_j}=\widetilde{w}_{\a_j}^{-1}\chi_{\a_i}\widetilde{w}_{\a_j}$$

Multiplying by $\widetilde{w}_{\a_i}$ and then $\widetilde{w}_{\a_j}$ on both sides gives the equality $\widetilde{w}_{\a_j}\widetilde{w}_{\a_i}\chi_{\a_j}=\chi_{\a_i}\widetilde{w}_{\a_j}\widetilde{w}_{\a_i}$, which is the image of $X_iS_jS_i=S_jS_iX_j$ under $\rho_\lambda$. Thus, the image of the relation $X_iS_jS_i=S_jS_iX_j$ is satisfied in $G^\lambda(\Z)$.

\medskip
(10) To determine the image of $S_i^2X_jS_i^{-2}$ under $\rho_\lambda$, note that $S_i^2=\widetilde{h}_i(-1)$, so that $$\rho_\lambda(S_i^2X_jS_i^{-2})=\rho_\lambda(\widetilde{h}_i(-1)X_j\widetilde{h}_i(-1)^{-1})=h_i(-1)\chi_{\a_j}(1)h_i(-1)^{-1}.$$
We have the relation $h_i(u)\chi_\a(v)h_i(u)^{-1}=\chi_\a(vu^{a_{ij}})$ in $G^\lambda(\Z)$ for any $i,j$ ([CER]). Hence, for $i,j$ adjacent, we see that $$\rho_\lambda(S_i^2X_jS_i^{-2})=h_i(-1)\chi_{\a_j}(1)h_i(-1)^{-1}=\chi_{\a_j}(-1)=\chi_{\a_j}^{-1}=\rho_\lambda(X_j)^{-1}.$$
Thus, the image of the relation $S_i^2X_jS_i^{-2}=X_j^{-1}$ for $i,j$ adjacent is satisfied in $G^\lambda(\Z)$.

\medskip
(11) For any prenilpotent pair of real roots $\alpha$ and $\beta$, by [Ti] and [CG] there is a Chevalley relation in $G^{\lambda}$ of the form

$$[\chi_{\alpha}(u),\chi_{\beta}(v)]
	=
	\prod_{\begin{matrix} m\alpha+n\beta\in\Delta^{re}_+\\m,n> 0\end{matrix}
	}
	\chi_{m\alpha+n\beta}(C_{mn\alpha\beta}u^m v^n)
$$
between root groups $U_{\alpha}$ and $U_{\beta}$. Here $u,v\in R$, $C_{mn\alpha\beta}$ are integers (structure constants). 
We claim that  adjacent simple roots $\alpha_i$ and $\alpha_j$ form a prenilpotent pair. By Lemma 6 of [AC], in the simply laced hyperbolic case, a pair of real roots is prenilpotent if $(\alpha,\beta)\geq -1$. But for adjacent simple roots $\alpha_i$ and $\alpha_j$, 
 $(\alpha_i,\alpha_j)=-1$. Thus the pair $\alpha_i,\alpha_j$ is prenilpotent and we have $C_{mn\alpha\beta}\in\{\pm 1\}$. We have 
$$(m\a_i+n\a_j \ | \ m\a_i+n\a_j)=m^2(\a_i \ | \ \a_i)+2mn(\a_i \ | \ \alpha_j)+n^2(\a_j \ | \ \alpha_j)=2m^2-2mn+2n^2.$$
But  $(\cdot \ | \ \cdot)=2$ for all elements of $\Delta^{re}$. To determine when $m\a_i+n\a_j \in\Delta^{re}$ for $m,n>0$, we seek non--negative solutions of 
$$2m^2-2mn+2n^2=2(m^2-mn+n^2)=2$$
or $$m^2-mn+n^2=1.$$ It is easy to see that the only solution is
$$m=n=1.$$
Thus the right hand side of the Chevalley commutation relation is $$\chi_{\a_i+\a_i}=\chi_{w_{\a_i}(\a_j)}=\widetilde{w}_{\a_i}\chi_{\a_j}\widetilde{w}_{\a_i}^{-1}$$ and
the image of the relation $[X_i,X_j]=S_iX_jS_i^{-1}$ for $i,j$ adjacent is satisfied in $G^\lambda(\Z)$.

\medskip
(12) Finally, consider $\rho_\lambda([X_i,S_iX_jS_i^{-1}])$ for $i,j$ adjacent. We have $$\rho_\lambda([X_i,S_iX_jS_i^{-1}])=[\rho_\lambda(X_i),\rho_\lambda(S_iX_jS_i^{-1})]=[\chi_{\alpha_i},\widetilde{w}_{\alpha_i}\chi_{\alpha_j}\widetilde{w}_{\alpha_i}^{-1}].$$
We have $\widetilde{w}_{\alpha_i}\chi_{\alpha_j}\widetilde{w}_{\alpha_i}^{-1}=\chi_{w_{\a_i}(\alpha_j)}$. As $i$ and $j$ are adjacent, $a_{ij}=-1$, thus it follows that   $w_{\a_i}(\alpha_j)=\alpha_i+\alpha_j$, so that $$\rho_\lambda([X_i,S_iX_jS_i^{-1}]=[\chi_{\alpha_i},\widetilde{w}_{\alpha_i}\chi_{\alpha_j}\widetilde{w}_{\alpha_i}^{-1}]=[\chi_{\alpha_i},\chi_{\alpha_i+\alpha_j}].$$
We note that  $[\chi_{\alpha_i},\chi_{\alpha_i+\alpha_j}]=1$ when $2\alpha_i+\alpha_j$ is not a root.
We have $$(2\a_i+\a_j \ | \ 2\a_i+\a_j)=4(\a_i \ | \ \a_i)+2(\a_i \ | \ \alpha_j)+2(\a_j \ | \ \alpha_i)+(\a_j \ | \ \alpha_j)=8-2-2+2=6.$$
As $(\cdot \ | \ \cdot)\leq 0$ for all elements of $\Delta^{im}$ and $(\cdot \ | \ \cdot)=2$ for all elements of $\Delta^{re}$, it follows that \linebreak $2\alpha_i+\alpha_j\not\in\Delta^{im}$ and $2\alpha_i+\alpha_j\not\in\Delta^{re}$, thus $2\alpha_i+\alpha_j$ is not a root.
Thus $$\rho_\lambda([X_i,S_iX_jS_i^{-1}])=[\chi_{\alpha_i},\chi_{\alpha_i+\alpha_j}]=1$$ and the image of the relation $[X_i,S_iX_jS_i^{-1}]=1$ for $i,j$ not adjacent, is satisfied in $G^\lambda(\Z)$.

\end{proof}

\newpage
\bigskip\section{The kernel of  $\rho_{\lambda}$}

When $R=\mathbb{K}$ is a field, R\'emy and Ronan took a completion $\Ghat(\mathbb{K})$ of Tits' Kac--Moody group $G(\mathbb{K})$ in the automorphism group of the building of the BN--pair for $G$ ([RR] and [Mar], Sec 6.1). 

The following relationship between $\widehat{G}^\lam(\mathbb{K})$ and $\Ghat(\mathbb{K})$
was established by [BR] and [Mar], (see also [CER]).

\begin{theorem}\label{kernel} For any regular weight $\lam$, there exists a (canonical)
continuous surjective \linebreak homomorphism $\rho_{\lam}:\widehat{G}^\lam(\mathbb{K})\to\Ghat(\mathbb{K})$.
The kernel  of $\rho_{\lam}$ is equal to $$\bigcap\limits_{g\in\Ghat}g\Bhat^{\lam}g^{-1}$$ and is contained in $Center(\widehat{G}^\lam(\mathbb{K}))$.
\end{theorem} 

This result was extended further by [Rou] (see also Prop 6.46 of [Mar]).

\begin{theorem} ([Rou]) If $Char(\mathbb{K})=0$  then $\eps_{\lam}:\widehat{G}^\lam(\mathbb{K})\to\Ghat(\mathbb{K})$ is an isomorphism of topological groups.
\end{theorem}

\bigskip
From now on, we will take $R=\C$.

\bigskip

\begin{lemma}\label{center} The kernel of the map $\rho_{\lam}: G(\C)\to G^{\lam}(\C)$
lies in $Center(G(\C))$.
\end{lemma}

\begin{proof} Consider the sequence of homomorphisms
\[
	\begin{CD}
	G(\C) @>\rho_\lam >> G^\lam(\C) @> \iota>> \Ghat^\lam(\C) @> \eps_\lam >> \Ghat(\C)
	\end{CD}
\]
where $\iota$ is the inclusion map. The composition of these three maps is the natural map from $G$ to $\Ghat$.
By Theorem~\ref{kernel},  $Ker(\eps_\lam)\leq Center(\Ghat^\lam(\C))$.  By ([RR], 1.B), 
$$Ker(\eps_\lam \circ\iota\circ \rho_\lam)=Center(G(\C))\leq H(\C)$$
where $$H(\C)=\langle\widetilde{h}_i(a):=\widetilde{s}_i(a)\widetilde{s}_i(-1)\mid a\in\C^{\times},\ \widetilde{h}_i(a)\cdot\widetilde{h}_i(b)=\widetilde{h}_i(ab)\rangle$$
for $\widetilde{s}_i(a):=X_i(a)S_iX_i(a^{-1})S_i^{-1}X_i(a).$ 
But $$Ker(\rho_{\lam})\leq Ker(\eps_\lam \circ\iota\circ \rho_\lam).$$
The result follows.
\end{proof}

\bigskip
Since  $\widetilde{h}_{\a_i}(1)=1$, we define   
$$H(\Z)=\langle \widetilde{h}_{\alpha_i}(-1)=\widetilde{s}_i(-1)^2\mid \widetilde{h}_{\alpha_i}(-1)^2=1,\ i\in\{1,\dots,\ell\}\rangle.$$
 Then $H(\Z)\cong(\Z/2\Z)^{rank(G)}$ and the natural map $H(\Z)\to H(\C)$ is injective.

\newpage

\newpage
 Injectivity of the natural map $G^{\lambda}(\Z)\to G^{\lambda}(\C)$ holds by definition. Injectivity of the natural map
$$\varphi:G(\Z)\to G(\C)$$
is  not currently known and depends on functorial properties of Tits' group ([Ti]).
\begin{lemma}\label{injective} Consider the  group homomorphism $\varphi:G(\Z)\to G(\C)$. Then $$\varphi(Center(G(\Z)))\subseteq Center(\varphi(G(\Z)))\leq Center(G(\C)).$$ \end{lemma}

\begin{proof}  
Consider $\varphi:G(\Z)\to G(\C)$. Then $$\varphi(Center(G(\Z)))\subseteq Center(\varphi(G(\Z))).$$ 
and $$K=\varphi(G(\Z))\cong G(\Z)/Ker(\varphi)$$ is a subgroup  of $G(\C)$, so $\varphi$ takes $Center(G(\Z))$ to the center of a subgroup $K\leq G(\C)$. But $Center(G(\C))$ is abelian so 
$$Center(K)=Center(\varphi(G(\Z)))\leq Center(G(\C)).\ 
\square$$
\end{proof}

\bigskip

The following diagram commutes:
\begin{figure}[h]
	\centering
		\includegraphics[scale=1.0]{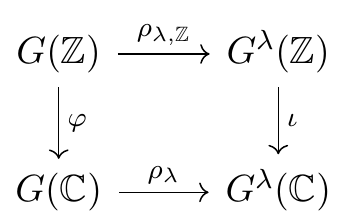}
\label{}
\end{figure}

That is, 
$$(\iota\circ \rho_{\lam,\Z}) (G(\Z))=(\rho_{\lam}\circ\varphi) (G(\Z))$$
so
$$Ker(\iota\circ \rho_{\lam,\Z}) =Ker(\rho_{\lam}\circ\varphi).$$

We  have
$$Ker(\rho_{\lam,\Z})\leq Ker(\iota\circ \rho_{\lam,\Z})=Ker(\rho_{\lam}\circ\varphi).$$
Thus if $\varphi$ is injective, then we have
$$Ker(\iota\circ \rho_{\lam,\Z}) =Ker(\rho_{\lam}\circ\iota)$$
or 
$$Ker(\rho_{\lam,\Z})\leq Ker(\iota\circ \rho_{\lam,\Z})=Ker(\rho_{\lam}\circ\iota)\leq Center(G(\C))\leq H(\C),$$
using Lemma~\ref{center}. Furthermore, if $\varphi$ is injective, then by Lemma~\ref{injective} we have $$Center(G(\Z))\leq Center(G(\C))$$
$$Ker(\rho_{\lam,\Z})\leq  Center(G(\Z))\leq Center(G(\C))\leq H(\C)$$
and $$Center(G(\Z))=G(\Z)\cap Center(G(\C))\leq G(\Z)\cap H(\C)=H(\Z).$$

We have proven the following.

\begin{theorem}  The kernel $K^{\lambda}$ of the map $\rho_{\lam,\Z}: G(\Z)\to G^{\lam}(\Z)$ lies in $H(\C)$ and if  the  group homomorphism $\varphi:G(\Z)\to G(\C)$ is injective, then $K^{\lambda}\leq H(\Z)\cong(\Z/2\Z)^{rank(G)}$.
\end{theorem}

\end{document}